\documentclass{amsart}
\usepackage{hyperref}
\usepackage{graphicx}
\usepackage{amscd}
\usepackage{amsmath}
\usepackage{amssymb}
\usepackage{verbatim}
\newtheorem{conjecture}{Conjecture}

\usepackage[dvipsnames]{xcolor}
\usepackage{tikz-cd} 
\usepackage{pgfplots}
\pgfplotsset{compat=newest}
\usetikzlibrary{arrows.meta}
\usetikzlibrary{bending}
\usepackage{bm}
\usepackage{asymptote}
\usepackage{dirtytalk}
\newtheorem{theorem}{Theorem}[section]
\newtheorem{lemma}[theorem]{Lemma}

\theoremstyle{definition}
\newtheorem{definition}[theorem]{Definition}
\newtheorem{proposition}[theorem]{Proposition}

\newtheorem{corollary}[theorem]{Corollary}

\theoremstyle{remark}
\newtheorem{remark}[theorem]{Remark}

\numberwithin{equation}{section}

\newcommand\isomto{\stackrel{\sim}{\smash{\longrightarrow}\rule{0pt}{0.4ex}}}
\begin{document}

\title{Known cases of the Hodge conjecture}

\author{Genival Da Silva Jr.}

\address{\newline Department of Mathematics \newline Eastern Illinois University\newline Charleston, IL}
\email{jrrieman@gmail.com}


\begin{abstract}
The Hodge conjecture is a major open problem in complex algebraic geometry. In this survey, we discuss the main cases where the conjecture is known, and also explain an approach by Griffiths-Green to solve the problem.
\end{abstract}

\maketitle

In his 1950 ICM address \cite{hodge}, Hodge proposed several questions concerning the topology and geometry of algebraic varieties. The majority of those problems were solved in the following years but one survived to this day. What now is known as \textit{the Hodge conjecture}, remains one of the more interesting, and deepest unsolved problem in complex algebraic geometry.

Hodge initially stated his question using integral cohomology classes. As stated, the conjecture would then be false. Atiyah and Hirzebruch \cite{atiyah} gave the first proof of this fact in the sixties. The rational coefficients version of his conjecture is what we know today as the Hodge conjecture.

The conjecture itself is a generalization of a result of Lefschetz \cite{LEFSCHETZ} proved years earlier, even before Hodge formulated his problem. The Lefschetz theorem on $(1,1)$-classes is the Hodge conjecture with integral coefficients in codimension one. The method of the proof employed by Lefschetz was to use a tool developed by Poincaré, called normal function. In modern algebraic geometry language, a more simple and elegant proof can be given, see \hyperref[sec1]{Section 1}.

Unfortunately, Lefschetz's proof can't be generalized to higher dimensions due to the fact that a crucial condition used in his proof, Jacobi inversion, can fail in codimension greater than one \cite{shioda2}. Despite this, a generalization of his technique was proposed by Phillip Griffiths \& Mark Green in the early 2000s, which uses normal functions as well, albeit differently, more on this in \hyperref[sec3]{Section 3}.

Over the years, many special cases of the conjecture were proven, most notably, many cases involving Abelian varieties. In most of cases, the idea is to use an additional structure of the projective variety. For example, the abelian condition, group action or symmetry of the defining equations, varieties with many lines on it, cohomology ring generated by divisors and so on.

The idea of this manuscript is to describe the main known cases of the Hodge conjecture and discuss a possible approach to the proof. The list is not meant to be a comprehensive list of every single case proved so far, but a place where the main cases can be found. The importance given to the examples are from the author's own personal taste and it's possible that relevant examples were left out of this text. Some topics related to the conjecture are not discussed here, like the General Hodge conjecture, the Tate conjecture, absolute Hodge classes, variational Hodge conjecture and the Standard conjectures.

In \hyperref[sec1]{Section 1} we introduce the notation of the paper, state the Hodge conjecture and discuss the Lefschetz's theorem on $(1,1)$-classes. \hyperref[sec2]{Section 2} reviews all known relevant cases of the conjecture. \hyperref[sec3]{Section 3} describes the Griffiths-Green program to solve the Hodge conjecture using singularities of normal functions. 

\subsection*{Acknowledgements} I thank Prof. Matt Kerr and Prof. Phillip Griffiths for many discussions regarding the Hodge conjecture and related topics.

\section{The statement of the problem}\label{sec1}
One of the motivations behind the Hodge conjecture was the following theorem by Lefschetz. 
\begin{theorem}(Lefschetz's theorem on $(1,1)$-classes)\label{lef11}
Let $X$ be a compact Kahler manifold, then the cycle class map $$c_1:H^1(X,\mathcal{O}^\times)\rightarrow H^{1,1}(X)\cap H^{2}(X,\mathbb{Z})$$ is surjective.
\end{theorem}
\begin{proof}
Consider the exponential exact sequence over $X$:
$$0\to\mathbb{Z}\xrightarrow{i} \mathcal{O}\xrightarrow{exp} \mathcal{O}^\times\to 0$$ where $i$ is the inclusion and $exp$ is the exponential map $exp(f)=e^f$. Taking cohomology we get:
$$...\to H^1(X,\mathcal{O}^\times) \xrightarrow{c_1} H^{2}(X,\mathbb{Z}) \xrightarrow{i_*} H^2(X,\mathcal{O})\to ...$$
Note that by definition of the Chern Class, we have $c_1(L)\in H^{1,1}(X)$ for any complex line bundle $L$. On the other hand, $\ker i_*$ is the image of $c_1$ and $H^2(X,\mathcal{O})\cong H^{0,2}(X)$, hence $i_*=0$ on $H^{1,1}(X)\cap H^{2}(X,\mathbb{Z})$ and $c_1$ is surjective.
\end{proof}
\begin{remark}
Note that $H^1(X,\mathcal{O}^\times)\cong Pic(X)$ and since $X$ is smooth, $Pic(X)\cong CH^1(X)$.
\end{remark}
Lefschetz proved his theorem only for surfaces and at the time of his proof, the theory of sheaf cohomology was not developed yet, so his approach to the problem was rather different from the one above.
His idea was to use a tool developed by Poincaré called \textbf{normal function}. To describe normal functions we need a few preliminaries first. Let $S$ be a smooth complex projective surface. A Lefschetz's pencil on $S$ is a family of hyperplane sections $\{C_t\}_{t\in\mathbb{P}^1}$ such that for finitely many $t\in\mathbb{P}^1$, $C_t$ has a unique singularity and it is a node. The existence of Lefschetz's pencil is not automatic from the definition, a proof can be found in \cite{voisin2,katz1,frei}.

After possibly blowing up the base locus, the Lefschetz's pencil gives a morphism $\pi:\tilde{S}\to \mathbb{P}^1$, where the fiber $C_t$ has a node for $t\in \{t_1,\dots,t_m\}$. The (extended) intermediate Jacobian bundle $\mathcal{J}_e$ is defined by
\begin{equation}\label{jacobian_e}
    0\to R^1\pi_*\mathbb{Z}\to R^1\pi_*\mathcal{O}_{\tilde{S}}\to \mathcal{J}_e\to 0 
\end{equation}
Poincaré then defined a normal function as a section of $\mathcal{J}_e$. 

An algebraic cycle $Z\in CH^1(S)$ defines a normal function $\nu_Z$ as follows. Let $Z_t:=Z\cdot C_t\in Div^\circ (C_t)$, then we define $\nu_Z(t)=AJ(Z_t)$, where $AJ$ is the Abel-Jacobi map.

Taking cohomology of \ref{jacobian_e} we get
\begin{equation}\label{classm}
0\to Pic^\circ(S)\to \Gamma(\mathbb{P}^1,\mathcal{J}_e)\xrightarrow{[.]} H^1(\mathbb{P}^1,R^1\pi_*\mathbb{Z})\cong H^2(S,\mathbb{Z})_{prim}\to H^1(\mathbb{P}^1,R^1\pi_*\mathcal{O}_{\tilde{S}})
\end{equation}
Lefschetz analyzed the map $[.]$ in \ref{classm} and proved that $\zeta=[\nu]$ for some normal function $\nu$, if and only if $\zeta\in H^{1,1}(X)\cap H^2(S,\mathbb{Z})_{prim}$. He also proved that $[\nu_Z]=[Z]$ for any cycle $Z$.

Poincare had proved earlier that under the above hypothesis, \textbf{Jacobi inversion} holds, which is the statement in this case that $\nu=\nu_Z$ for some algebraic cycle $Z$, which completes Lefschetz's original proof of his theorem.

Unfortunately, Jacobi inversion does not hold in general, so Lefschetz's approach can't be generalized, at least not using Jacobi inversion to generalize Poincaré's result. In early 2000's, Griffiths and Green \cite{gg} gave an approach to prove the Hodge conjecture inductively, using normal functions. We'll discuss their program in \hyperref[sec3]{Section 3}.

The modern statement of the Hodge conjecture is the following:

\begin{conjecture}\label{hc}
Let $X$ be smooth complex projective variety of dimension $n$ and $p$ a number with $p\leq n$. The (rational) cycle class map
 $$cl_{\otimes\mathbb{Q}}:CH^p(X)\otimes\mathbb{Q}\rightarrow H^{p,p}(X)\cap H^{2p}(X,\mathbb{Q})$$ is surjective.
\end{conjecture}
In our notation, $cl_{\otimes\mathbb{Q}}(\sum a_iX_i) = \sum a_i[X_i]$, $a_i\in\mathbb{Q}$ and $[X_i]$ is the class of the subvariety $X_i$. 

\begin{remark}
Note that Lefschetz's theorem is true over $\mathbb{Z}$, but in higher codimension, the $\otimes\mathbb{Q}$ is necessary. Atiyah \& Hirzebruch \cite{atiyah} constructed a counter-example for torsion hodge classes.
\end{remark}
\begin{remark}
The fact that $X$ is not only compact Kahler but is projective is also necessary. Zucker \cite{zucker} gave a counter-example in case of a complex torus that wasn't an algebraic variety.
\end{remark}
\begin{remark}
Hodge also stated another conjecture on algebraic cycles. Known today as the \textbf{Hodge General Conjecture}, it is also about cohomology classes and algebraic subvarieties but much more general than the conjecture we are discussing here, in particular it implies conjecture \ref{hc}. As stated by Hodge, his general conjecture was false, Grothendieck gave a counter-example and proposed an amended version of it in \cite{gro1}.
\end{remark}

\section{Known cases of the conjecture}\label{sec2}
\textbf{Unless otherwise stated, $X$ will be a smooth complex projective variety.} All results stated here are well known and most of them were proved many decades ago. See also \cite{shioda2,lewis}.

Recall  the following results of Lefschetz:
\begin{theorem}(Lefschetz's theorem on hyperplane sections)\label{lh}
Let $i:Y\to X$ be a inclusion of a hyperplane section in $X$. Then the natural map
$$i^*:H^k(X,\mathbb{Z})\to H^k(Y,\mathbb{Z})$$
is an isomorphism for $k<n-1$ and is injective if $k=n-1$.
\end{theorem}
\begin{theorem}(Hard Lefschetz)\label{hl}
Let $X$ be a compact Kahler manifold with Kahler class $[\omega]\in H^{1,1}(X)\cap H^2(X,\mathbb{Z})$. The map on forms $L(x)=x\wedge [\omega]$ descend to cohomology. For $k\leq n$ we have an isomorphism
$$L^{n-k}:H^{k}(X,\mathbb{Q})\xrightarrow{\sim} H^{2n-k}(X,\mathbb{Q})$$
which is a morphism of Hodge structures of type $(1,1)$.
\end{theorem}
\begin{remark}
A class $\alpha\in H^{k}(X,\mathbb{Q})$ is primitive if $L(\alpha)=0$.
\end{remark}
Since a morphism of Hodge structures preserves algebraic cycles we deduce
\begin{corollary}
For $2p\leq n$, if the Hodge conjecture is true in codimension $p$, then it's also true in codimension $n-p$.
\end{corollary}
By theorem \ref{lef11}, we get
\begin{corollary}
The Hodge conjecture is true in codimension $n-1$.
\end{corollary}
Using the above we arrive at
\begin{theorem}\label{threefolds}
The Hodge conjecture is true for all $X$ with dimension less than or equal to $3$.
\end{theorem}
\begin{proof}
In dimension $1$, we have $H^0(X,\mathbb{Q})\cong [X],H^2(X,\mathbb{Q})\cong [pt]$, where $\cong$ means generated by. If $X$ has dimension 2 then $H^0(X,\mathbb{Q})\cong [X],H^4(X,\mathbb{Q})\cong [pt]$ and $H^{1,1}(X)\cap H^2(X,\mathbb{Q})$ is algebraic by theorem \ref{lef11}. In dimension $3$, we have $H^0(X,\mathbb{Q})\cong [X]$, $H^{1,1}(X)\cap H^2(X,\mathbb{Q})$ algebraic, and the Hodge classes in $H^4,H^6$ are algebraic by the above corollary.
\end{proof}
\begin{proposition}
The Hodge conjecture is true if $X$ is a flag variety, in particular, it's true for all Grassmanians $G(a,b)$.
\end{proposition}
\begin{proof}
The integral cohomology ring of a flag variety is generated by Schubert cycles, and the latter are algebraic by construction.
\end{proof}
\subsection*{Hypersurfaces of degree $d$ in $\mathbb{P}^{n+1}$}
Note that any hypersurface can be seen as a hyperplane section of $\mathbb{P}^{n+1}$, hence we can apply theorem \ref{lh}. Since the Hodge conjecture is trivially true for projective spaces, the only hodge classes that are left to check are the ones in $H^{n,n}\cap H^{2n}(X,\mathbb{Q})$. 
\begin{theorem}
The Hodge conjecture is true for hypersurfaces of degree 1 and 2, linear varieties and quadrics respectively. 
\end{theorem}
\begin{proof}
Recall that a variety has degree 1 if and only if it's linear, hence its cohomology is the same as $\mathbb{P}^{n}$. In the degree 2 case, we have that the cohomology ring has a cellular decomposition and is again generated by Schubert cycles \cite{gh}, hence Hodge classes are algebraic in this setting as well.
\end{proof}
\begin{remark}
This proposition summarizes all known cases of the Hodge conjecture for hypersurfaces based on degree only, which don't depend on the dimension $n$. 
\end{remark}
We now discuss the case of fourfolds hypersurfaces.
\begin{lemma}\label{proj_bundle}
Let $p:P(E)\to X$ be a projective bundle. If the Hodge conjecture is true for $X$ then it's also true for $P(E)$. In particular, it's also true for Flag Bundles.
\end{lemma}
\begin{proof}
The cohomology ring $H^*(P(E))$ is an algebra over $H^*(X)$ generated by $c_1(\mathcal{O}_{P(E)}(1))$, with coefficients in the Chern classes of $E$. By construction, Chern classes come from algebraic cycles and the result follows.
\end{proof}
\begin{lemma}\label{fin_d}
Let $f:X\to Y$ be a morphism of degree $d>0$ between smooth complex algebraic varieties. If the Hodge conjecture holds for $X$ then it holds for $Y$.
\end{lemma}
\begin{proof}
Let $a\in H^{p,p}(Y,\mathbb{Q})$ be a Hodge class, then the pullback $f^*(a)$ is a Hodge class in $X$, so if the Hodge conjecture is true, $f^*(a)=[Z]$ for some algebraic cycle $Z$. Then $da=f_*f^*(a)=f_*[Z]$, so $a=\frac{1}{d}f_*[Z]$.
\end{proof}
\begin{theorem}(Zucker \cite{zucker})
The Hodge conjecture is true for cubic fourfolds.
\end{theorem}
\begin{proof}
The idea of the Proof is to use the relative Fano variety of lines, which in this case is a threefold, and use proposition \ref{threefolds} and the lemmas above.

Let $\{X_t\}_{t\in\mathbb{P}^1}$ be a Lefschetz pencil. A general member of this pencil is a smooth cubic threefold. Cubic threefolds were extensively studied in \cite{gc}. Consider the relative Fano variety of lines
$$F:=\{ (t,l)\in \mathbb{P}^1\times G(2,5) \;|\; l\subseteq X_t \}$$
and the tautological projective line bundle over it
$$E:=\{ (x,t,l)\in X\times F \;|\; x\subseteq l \subseteq X_t \}$$
so that we have natural projections map $\pi:E\to F$ and $\mu:E\to X$. After blowing up the base locus, if necessary, we can assume $F$ smooth. By \cite{gc}, $\mu$ has degree $6$ and we also have $\dim F=3$ and $E$ a projective bundle. The result then follows from the two lemmas above.
\end{proof}
\begin{remark}
Zucker also gave another, more technical, proof using normal functions. The disadvantage of the proof above, according to him, is that it can't be generalized, since it used the fact that the dimension of $F$ is $3$ and for higher degree or dimension the dimension should be higher.
\end{remark}
\begin{remark}
That is not to say that his proof using normal functions can be generalized either, as we noted before, Jacobi inversion can fail in codimension greater than one \cite{shioda1}. So a different method must be used in order to arrive at a generalization of Poincaré existence theorem. 
\end{remark}
We now discuss the unirational case, but first a lemma (see \cite[chapter 13]{lewis} for a proof).
\begin{lemma}
Let $D\subseteq X$ be a smooth irreducible subvariety of codimension $r\geq 2$ and $B_D(X)$ the blow-up of $X$ along $D$ with exceptional divisor $E$ and morphism $f:B_D(X)\to X$. Then
$$H^k(B_D(X),\mathbb{Q})=H^k(X,\mathbb{Q})\oplus H^{k-2}(D,\mathbb{Q})\oplus \ldots \oplus H^{k-2r+2}(D,\mathbb{Q})$$
\end{lemma}
Notice that in particular if the Hodge conjecture holds for $X$ and $D$ then it holds for $B_D(X)$, hence the following is immediate:
\begin{lemma}\label{hc_blow}
If $D\subset X$ is smooth irreducible subvariety of codimension $r\geq 2$ and $\dim D\leq 3$, then if the Hodge conjecture holds for $X$, it also holds for $B_D(X)$.
\end{lemma}
We are ready to prove the next case.
\begin{theorem}(Murre \cite{murre})\label{murre}
The Hodge conjecture is true for unirational fourfolds.
\end{theorem}
\begin{proof}
By hypothesis there is a dominant rational map $f:\mathbb{P}^4\to X$. By hironaka's theorem, we can find $Y$ obtained from $\mathbb{P}^4$ after a sequence of blow ups of surfaces, lines and points, such that there is a proper surjective morphism of finite degree $m:Y\to X$. The result follows from the previous lemma and lemma \ref{fin_d}.
\end{proof}
\begin{definition}
We say $X$ is uniruled if $X$ can be covered by lines. Equivalently, there's a dominant rational map $Y\times \mathbb{P}^1\to X$, where $\dim Y=\dim X -1$.
\end{definition}
Slightly adjusting the argument from the previous theorem we get
\begin{theorem}(Conte-Murre \cite{conte_murre})
The Hodge conjecture is true for uniruled fourfolds.
\end{theorem}
\begin{proof}
By hypothesis, there is a dominant rational map $Y\times \mathbb{P}^1\to X$ where $\dim Y=3$. Note that the Hodge conjecture is true for $Y\times \mathbb{P}^1$. Indeed, the Kunneth's formula respects Hodge types. Now we repeat the argument of theorem \ref{murre}. By Hironaka's theorem we can find $W$ smooth and obtained by blowups, and proper surjective morphism of finite degree $m:W\to X$. The result follows from the lemma \ref{hc_blow} and lemma \ref{fin_d}.
\end{proof}
\begin{corollary}
The Hodge conjecture is true for quartic and quintic fourfolds.
\end{corollary}
\begin{proof}
Quartic and quintic fourfolds are uniruled \cite{conte_murre}.
\end{proof}
After more than 40 years, a complete proof for fourfolds of degree six and beyond have not been found, except for particular cases. In fact, some fourfolds are even candidates for counter-example to the Hodge conjecture, as we shall see soon.

We now discuss complete intersection fourfolds. The method of proof is very different from the ones presented so far. Instead of using lines or blowups, the idea consists of using a theorem by Bloch and Srinivas, which concerns the class of the diagonal in $X\times X$.
\begin{theorem}(Bloch-Srinivas \cite{bloch-s})
Let $Y$ be a smooth complex projective variety of dimension $n$ such that $CH_0(Y)$ is supported on a closed algebraic subset $W\subseteq Y$. Then if $\Delta\in CH^n(Y\times Y)$ is the class of the diagonal, there is a nonzero integer $N$, $Z_1,Z_2\in CH^n(Y\times Y)$, and a divisor $D$ such that:
$$N\Delta=Z_1+Z_2$$
and $Supp \;Z_1=D\times Y,\quad D\subsetneq Y,\quad Supp \;Z_2=Y\times W$
\end{theorem}
An important corollary from the theorem above is the following:
\begin{corollary}(Bloch-Srinivas \cite{bloch-s})
Let $X$ be a smooth complex projective variety such that $CH_0(X)$ is supported on a closed algebraic subset $X'\subseteq X$ with $\dim X'\leq 3$. Then the Hodge conjecture is true for $(2,2)$-classes on $X$.
\end{corollary}
\begin{proof}
By the theorem above, there exists a divisor $D$, an integer $N$, and cycles $Z_1,Z_2$ such that, after taking cohomology, we have:
$$N[\Delta] = [Z_1]+[Z_2] \text{ in } H^{2n}(X\times X,\mathbb{Z})$$
Let $k:\tilde{D}\to X,j:\tilde{X'}\to X$ be the inclusion of the desingularization of $D$ and $X'$, respectively. Each Kunneth components of $H^{2n}(X\times X,\mathbb{Z})$ can be seen as a morphism of Hodge structures, in particular taking the components of type $(4,2n-4)$ we have:
$$N[\Delta]^* = [Z_1]^*+[Z_2]^* \;: H^4(X,\mathbb{Z})\to H^4(X,\mathbb{Z})$$
Now let $Z_1'\subset \tilde{D}\times X$ be a codimension $n$ cycle such that 
$$(k,Id)_*(Z_1')=Z_1$$
Similarly, let $Z_2'\subset X\times \tilde{X'}$ be such that
$$(Id,j)_*(Z_2')=Z_2$$
then it follows that for any $\alpha\in H^4(X,\mathbb{Z})$ we have
$$[Z_1]^*(\alpha)=k_*([Z_1']^*(\alpha))$$
and
$$[Z_2]^*(\alpha)=[Z_2']^*(j^*(\alpha))$$
If $\alpha\in H^{2,2}(X)\cap H^4(X,\mathbb{Q})$, then $j^*(\alpha)$ is in $H^{2,2}(\tilde{X'})\cap H^4(\tilde{X'},\mathbb{Q})$ and $[Z_1']^*(\alpha)$ is in $H^{1,1}(\tilde{D})\cap H^2(\tilde{D},\mathbb{Q})$. The relation $$N[\Delta]^*(\alpha) = N\alpha = k_*([Z_1']^*(\alpha))+[Z_2']^*(j^*(\alpha))$$
shows that $\alpha$ is algebraic since the Hodge conjecture is known in dimension less than $4$, see theorem \ref{threefolds}.
\end{proof}
Now a theorem of Roitman \cite{roitman} proves that if a smooth complete intersection $X$ has geometric genus zero then $A_0(X)=0$, in particular the corollary above applies. 

Let $H_1,\ldots, H_k$ be hypersurfaces in $\mathbb{P}^n$ of degrees $d_1,\ldots,d_k$ respectively. If $d_1+\ldots+d_k\leq n$ then the geometric genus of a smooth irreducible component of $H_1\cap \ldots \cap H_k$ is zero. Therefore, by the above we have:
\begin{theorem}(Bloch-Srinivas \cite{bloch-s})
The Hodge conjecture is true for the following fourfolds:
\begin{itemize}
    \item cubic, quartic and quintics in $\mathbb{P}^5$
    \item intersections of two quadrics, a quadric and a cubic, two cubics, or a quartic and a quadric in $\mathbb{P}^6$
    \item intersections of two quadrics and a cubic or three quadrics in $\mathbb{P}^7$
    \item intersection of 4 quadrics in $\mathbb{P}^8$
\end{itemize}
\end{theorem}
\subsection*{Fermat varieties $X_m^n$} Fermat varieties of degree $m$ and dimension $n$ are hypersurfaces $X_m^n$ in $\mathbb{P}^{n+1}$ defined by the zeros of
$$x_0^m+\ldots+x_n^m=0$$
T. Shioda \cite{shioda3} pioneered the study of the Hodge conjecture for Fermat varieties, and described an algorithm to prove the conjecture for all Fermat varieties, reducing the problem to a counting problem. Unfortunately, his program fails in some cases, and a complete proof for Fermat Varieties is still lacking to this day. 

The main theorem in the context of Fermat varieties is the following fact:
\begin{theorem}(Shioda \cite{shioda3})\label{shioda3}
    Let $n=r+s$ with $r,s\geq 1$. Then there is an isomorphism 
    $$f: [H^r_{prim}(X^r_m,\mathbb{C})\otimes H^s_{prim}(X^s_m,\mathbb{C})]^{\mu_m} \oplus H^{r-1}_{prim}(X^{r-1}_m,\mathbb{C})\otimes H^{s-1}_{prim}(X^{s-1}_m,\mathbb{C})\isomto H^n_{prim}(X^n_m,\mathbb{C})$$
    with the following properties:
    \begin{itemize}
        \item[a)] $f$ is morphism of Hodge structures of type (0,0) on the first summand and of type $(1,1)$ on the second.
        \item[b)]If $n=2p$ then $f$ preserves algebraic cycles, moreover if $$Z_1\otimes Z_2\in H^{r-1}_{prim}(X^{r-1}_m,\mathbb{C})\otimes H^{s-1}_{prim}(X^{s-1}_m,\mathbb{C})$$
        then $f(Z_1\otimes Z_2)=mZ_1\wedge Z_2$, where $Z_1\wedge Z_2$ is the algebraic cycle obtained by joining $Z_1$ and $Z_2$ by lines on $X^n_m$, when $Z_1,Z_2$ are viewed as cycles in $X^n_m$.
    \end{itemize}
\end{theorem}
Using the above theorem and the fact that the cohomology of Fermat varieties decomposes as a direct sum of Eigenspaces for finite group action, Shioda was able to reduce the problem to verifying a condition on a finite number of elements of a semi-group. His result was the following:
\begin{theorem}(Shioda \cite{shioda3})
    The Hodge conjeture is true for Fermat varieties $X^n_m$ with degree less than $21$.
\end{theorem}
For $X^n_m$ of prime degree, we actually can say more:
\begin{theorem}(Shioda, Ran \cite{shioda3,ran})
    If $m$ is prime or $m=4$, the cohomology ring of $X^n_m$ is generated by linear subspaces, thus the Hodge conjecture is true for $X^n_m$.
\end{theorem}
Very recently, using Shioda's results I was able to prove
\begin{theorem}(da Silva \cite{jr1})
If $m$ is coprime to $6$, then the Hodge conjecture is true for all Fermat fourfolds $X^4_m$. Also, a computer verification proves that the conjecture is also true  for $X^n_{21}$ and $X^n_{27}$.
\end{theorem}
An interesting case is the Fermat $X^n_{33}$. A simple computer verification \cite{jr1} shows that Shioda's program fails in this case. Hence, there are Hodge cycles not coming from the induced structure in this case, and we can't use Shioda's method to prove the Hodge conjecture for all Fermat varieties. Interestingly, his program also fails for $X^n_{p^2}$ for $p$ prime, but N. Aoki \cite{aoki} proved the Hodge conjecture in this case:
\begin{theorem}(Aoki \cite{aoki})
The Hodge conjecture is true for $X^n_{p^2}$, $p$ prime.
\end{theorem}
In this case, the cycles that do not come from the induced structure can be explicitly constructed. Sadly, it's unlikely his proof can be generalized in all degrees. 

Shioda's results also proved a special case of the Hodge conjecture for Abelian varieties, namely, those of Fermat type.
\begin{definition}
An abelian variety $A$ is said to be of Fermat type of degree $m$ if it is isogeneous to a product $J_1\times \ldots \times J_k$ of Jacobian of curves dominated by the Fermat curve $X^1_m$.
\end{definition}
A simple application of the discussion above gives:
\begin{theorem}(Shioda,\cite{shioda3})
If the Hodge conjecture is true for the Fermat $X^n_m$ regardless of $n$, then it's true for any abelian variety $A$ of Fermat type of degree $m$.
\end{theorem}
\begin{proof}
Indeed, by definition of Jacobian, we can find a morphism of finite degree $f:X^1_m\times \ldots \times X^1_m\to A$. Now the Hodge conjecture is true for the product $X^1_m\times \ldots \times X^1_m$ \cite[Theorem IV]{shioda3}. The result then follows from lemma \ref{fin_d}.
\end{proof}
\begin{remark}
This proof drastically differs from the \say{usual} proof of the Hodge conjecture for Abelian varieties --as we'll see next-- in the sense that the majority of cases is proven by verifying the fact that the cohomology is generated by the intersection of divisors. In particular, for Abelian varities of Fermat type the conjecture is true even in cases where the cohomology ring is not generated by divisors.
\end{remark}
\subsection*{Abelian varieties}

The case of general Abelian varieties is a interesting one, in the sense that besides being smooth complex and projective, it's also an abelian group, so techniques from group theory and representation theory are available. 

In this subsection $A$ will always denote an Abelian variety.

\begin{definition}
Let $B^p(A)$ be the group of codimension $p$ Hodge cycles, $C^p(A)$ be the group of codimension $p$ algebraic cyles. We set $D^p(A)=\bigwedge_p C^1(A)$, the subgroup generated by divisors. 
\end{definition}
Note that by Lefschetz's $(1,1)$ theorem, $B^1(A)=C^1(A)$ and
$$D^p(A) \subseteq C^p(A) \subseteq B^p(A)$$
As we said above, in practically all cases of Abelian varieties the method of proof of the Hodge conjecture is the same, namely, show that $D^p(A)= B^p(A)$ and hence $C^p(A) = B^p(A)$.
\begin{theorem}The condition $D^p(A)= B^p(A)$ is true for the following class of abelian varieties:
\begin{itemize}
    \item[i)] (Tate,Murasaki \cite{tate,murasaki}) $E^n$, self product of an elliptic curve.
    \item[ii)] (Tankeev, Ribet \cite{tankeev,ribet}) Simple abelian variety of prime dimension.
    \item[iii)] (Hall,Kuga \cite{kuga-hall}) Generic fibers of a certain family of abelian varieties.
    \item[iv)] (Mattuck \cite{mattuck}) Certain abelian varieties with a condition on the period matrix.
\end{itemize}

\end{theorem}
\begin{remark}
Tate announced item $i)$ above in a 1964 Seminary, but did not give a proof. A complete proof was given by Murasaki \cite{murasaki} 4 years later. The idea of the proof is simple, write everything explicitly and show that $D^p(E^n)= B^p(E^n)$ by induction on $p$. This proof can't naturally be generalized to all abelian varieties because some abelian varieties have $D^p(A)\neq B^p(A)$ and explicit cohomology computations tends to be very hard to do in the general case.
\end{remark}
\begin{definition}
A Hodge cycle $\alpha\in B^p(A)$ is said to be \textit{exceptional} if it's not in $D^p(A)$, i.e it's not an intersection of divisors.
\end{definition}
As discussed above, some Abelian varieties of Fermat type have exceptional Hodge classes, and the Hodge conjecture is still true. On the other hand, Mumford \cite{pohl}  gave an example of an abelian fourfold of CM type with an exceptional Hodge class for which, after 50 years, we still don't know if it is algebraic. 

Like Hodge cycles that don't come from the induced structure in Fermat varieties, exceptional Hodge classes in abelian varieties are candidates for counter-example to the Hodge conjecture. A complete proof that exceptional hodge classes are algebraic is what is still missing to solve the Hodge conjecture for all Abelian varieties. Certainly, such a proof is as difficult to find as it is to prove the Hodge conjecture for any $X$, abelian or not.
\section{The Griffiths-Green program}\label{sec3}
Around 2006, Griffiths and Green \cite{gg} outlined a method that potentially could prove the Hodge conjecture. The aim of this section is to describe their ideas and discuss some examples. To learn more about this approach please see Griffiths-Green's original paper \cite{gg}, Brosnan et al \cite{patrick} and the wonderful survey \cite{matt}.

Note that if one wants to prove the Hodge conjecture by induction on dimension, then it's enough to assume that $X$ has dimension $n = 2m$ and prove the Hodge conejcture for $(m,m)$-classes, that is $p=m$. Indeed, by Hard Lefschetz (theorem \ref{hl}), it's enough to prove for $(p,p)$-classes where $2p\leq n$. Suppose $2p<n$, let $k:=n-2p$ and $Y$ be a $2p$-dimensional subvariety of $X$ obtained by the complete intersection of $k$ general hyperplane sections of $X$, we can assume by induction that the Hodge conjecture is true for $Y$. Now, let $\alpha\in H^{p,p}(X,\mathbb{Q})$ be a Hodge class, and $i:Y\to X$ the inclusion map, then $i^*(\alpha)$ is an algebraic hodge class $[Z_Y]$ in $Y$. If we let $Y$ varies in family of complete intersections covering $X$, $[Z_Y]$ glue together to an algebraic cycle $[Z]$ on $X$, with the property that $i^*([Z])=i^*(\alpha)$, but by Lefschetz's hyperplane section theorem (\ref{hl}), $i^*$ is injective on $(p,p)$ classes, hence $[Z]=\alpha$.

From now on, we assume that $X$ has dimension $n=2p$ and we are after the Hodge conjecture for $(p,p)$-classes.
Let $Y$ be a hyperplane section of $X$ with inclusion $i:Y\to X$ and $L_X$ the Lefschetz operator, which is the cup product with the Kahler class of $X$. Then we have the following diagram:

\[
\begin{tikzcd}[row sep=3em]
 & H^{n-2}(Y,\mathbb{Q}) \arrow[dr,"i_*"] \\
H^{n-2}(X,\mathbb{Q}) \arrow[ur, "\sim"] \arrow[rr,"L_X"] && H^n(X,\mathbb{Q}) 
\end{tikzcd}
\]

The isomorphism is by weak lefschetz (\ref{lh}). Then by Lefschetz decomposition theorem we have $H^n(X,\mathbb{Q})=Prim^n(X,\mathbb{Q})\oplus i_*(H^{n-2}(Y,\mathbb{Q}))$, and on the level of Hodge components we have:
$$H^{p,p}(X,\mathbb{Q})=Prim^{p,p}(X,\mathbb{Q})\oplus i_*(H^{p-1,p-1}(Y,\mathbb{Q}))$$

Therefore, if we want to use induction to prove the Hodge conjecture it's enough to prove the result for primitive Hodge classes $\alpha\in Prim^{p,p}(X,\mathbb{Q})$.

Set $Hg^p(X):=Prim^{p,p}(X,\mathbb{Q})$ and suppose we are given a very ample line bundle $L\to X$ whose first Chern class $c_1(L)$ is in the class of the Kahler form of $X$.

Let $\bar{S}:=|L|$ the complete linear system of divisors associated with $L$, $\mathcal{X}$ the incidence variety
$$\mathcal{X}=\{(x,s)\in X\times \bar{S}\; |\; s(x)=0\}$$
with natural morphism $\pi:\mathcal{X}\to \bar{S}$, and $\hat{X}$ be the dual variety of $X$ (points $s$ in $\bar{S}$ such that $X_s=\pi^{-1}(s)$ is singular). Let $\mathcal{H}$ be the variation of Hodge structures over $S:=\bar{S}-\hat{X}$ associated to the local system $\mathbb{H}:=R^{2p-1}\pi^{sm}_*\mathbb{Z}(m)$.

Recall from \hyperref[sec1]{Section 1} that a (extended) normal function is a section of the Jacocian bundle $\mathcal{J}_e$, see \ref{jacobian_e}. In our setting, the $X_t$ are given by zero sections of $L$ and are not necessarily a Lefschetz pencil, so we generalize the definition as follows.
\begin{definition}
Let $\mathcal{H}$ be a variation of polarized Hodge structures of weight $-1$ over $S$ and $J(\mathcal{H})$ be the associated Jacobian bundle. We have the following exact sequence:
$$0\to \mathbb{H} \to \mathcal{H}\slash \mathcal{F}^0\to J(\mathcal{H}) \to 0$$
a normal function is a section of $J(\mathcal{H})$.
\end{definition}
As before (see \hyperref[sec1]{Section 1}), taking cohomology of the above sequence, we see there's a map $$[.]:H^0(S,J(\mathcal{H}))\to H^1(S,\mathbb{Z})$$
Let's assume now that $S$ is a quasi-projective variety inside a smooth projective variety $\bar{S}$ with inclusion $i:S\to\bar{S}$.
\begin{definition}
Let $\nu\in H^0(S,J(\mathcal{H}))$ be a normal function. Then the singularity of $\nu$ at $s\in \bar{S}$, denoted by $\sigma_s(\nu)$, is the colimit:
$$\sigma_s(\nu)=\varinjlim_{s\in U} \;[\nu_{U\cap S}]\in (R^1i_*\mathbb{H})_s$$
The natural image of $\sigma_s(\nu)$ in cohomology with rational coefficients will be denoted by $sing_s(\nu)$, and we say $\nu$ is singular in $\bar{S}$ if $sing_s(\nu)\neq 0$ for some $s\in\bar{S}$.
\end{definition}
The variations of Hodge structures that we've considered so far actually satisfy certain additional geometric conditions called admissibility conditions, see \cite{saito,patrick} for definitions. The resulting normal functions are then called admissible normal functions, and we denote by $ANF(S,\mathcal{H})$ the set of all such objects.

Using the above notation, Griffiths and Green \cite{gg} conjectured the following:
\begin{conjecture}\label{gg_c}
Let $\alpha\in Hg^p(X)$ and $\nu_\alpha$ be the associated admissible normal function. Then $\nu_\alpha$ is singular on $\bar{S}=|L^{\otimes d}|$ for some $d>0$.
\end{conjecture}
Interestingly, this conjecture is equivalent to the Hodge conjecture:
\begin{theorem}(Griffiths-Green \cite{gg,patrick,matt}\label{gg_theorem}
Conjecture \ref{gg_c} holds if and only if the Hodge conjecture holds.
\end{theorem}
\begin{proof}(Sketch \cite{matt}) The proof of this theorem is somewhat technical, so we only sketch the main ideas involved.

Fix $s\in\hat{X}$ and consider the following diagram:
\[ \begin{tikzcd}[row sep=3em]\label{gg_cd}
Hg^p(X) \arrow{r}{\nu_{(.)}} \arrow[swap]{d}{\alpha_s} & ANF(S,\mathcal{H})\slash J^p(X) \arrow{d}{sing_s(.)} \\%
H^{2p}(X_s) \arrow{r}{\beta_s} & (R^1j_*\mathbb{H}_\mathbb{Q})_s
\end{tikzcd}
\]
where $\alpha_s$ is the restriction and $\beta_s$ is a map that makes the diagram commute, it is described in details in \cite{patrick}, and is injective on the image of $\alpha_s$ (after possibly changing $L$ by $L^{\otimes d}$).

Richard Thomas proved in \cite{thomas} that if the Hodge conjecture is true, every primitive hodge class $\zeta\in Hg^p(X)$ restricts non-trivially to a hypersurface in $X$. If that us true then using the diagram \ref{gg_cd}, we see that for some  $s\in\hat{X}$, $\beta_s(\alpha_s(\zeta))$ is not zero, hence $sing_s(\nu_\zeta)$ is not zero.

Conversely, suppose $sing_s(\nu_\zeta)\neq 0$ for some $s\in\hat{X}$, then $\zeta$ restricts non-trivially in $X_s$. Let $\tilde{X_s}$ the desingularization of $X_s$ and $g:\tilde{X_s}\to X$ the natural inclusion. Then $g^*(\zeta)\neq 0$, by Poincaré duality there is a Hodge class $\xi\in H^{p-1,p-1}(\tilde{X},\mathbb{Q})$ such that $\xi\cup g^*(\zeta)\neq 0 $. Using the exact same argument we used in the beggining of this section, we can cover $\tilde{X}$ with hyperplane sections, assume the Hodge conjecture for the sections and produce an algebraic cycle $W$ on $\tilde{X}$ such that $[W]=\xi$, but then $g^*(\zeta)\cup \xi=g^*(\zeta)\cup [W]\neq 0$, by the projection formula we have $\zeta\cup g_*[W]\neq 0$, by Poincaré duality, $\zeta$ has to be algebraic.
\end{proof}
\begin{remark}
Not a single new case of the Hodge conjecture has been proved so far using theorem \ref{gg_theorem}. But examples of this approach to show known cases were given in \cite{gg}.
\end{remark}
\begin{remark}
A natural question is if conjecture \ref{gg_c} is true, what is the minimal $d>0$ for a fixed $L$? Is $d$ related to any topological or algebraic invariant of $X$ itself?
\end{remark}
\begin{remark}
A interpretation of the General Hodge Conjecture in this setting is discussed in \cite{matt-greg}
\end{remark}
\bibliographystyle{amsplain}

\end{document}